\documentclass{article}
\usepackage{amsfonts}
\usepackage{mathrsfs}
\usepackage{amsthm}
\usepackage{amssymb}
\usepackage{amsmath}
\usepackage{enumerate}
\usepackage{pb-diagram}
\usepackage{amscd}
\usepackage{colortbl}

\theoremstyle{plain}
\newtheorem{theorem}{Theorem}[section]
\newtheorem{proposition}[theorem]{Proposition}
\newtheorem{corollary}[theorem]{Corollary}
\newtheorem{lemma}[theorem]{Lemma}

\theoremstyle{definition}
\newtheorem{definition}{Definition}[section]

\theoremstyle{remark}
\newtheorem{remark}{\textbf{Remark}}[section]

\theoremstyle{example}

\numberwithin{equation}{section}

\title{Generalized weighted number operators on functionals of discrete-time normal martingales}
\author{Jing Zhang $^{1,2}$,\ \  Caishi Wang $^{1,}$\footnote{Author to whom correspondence should be addressed.
Electronic addresses: wangcs@nwnu.edu.cn},\ \  Lixia Zhang $^1$,\ \  Lu Zhang $^1$\\
$^1$ School of Mathematics and Statistics,
Northwest Normal University\\
Lanzhou 730070, People's Republic of China\\
$^2$ College of Science, Gansu Agricultural University\\
Lanzhou 730070, People's Republic of China}

\begin{document}
\maketitle

\noindent\textbf{Abstract.}\ \
Let $M$ be a discrete-time normal martingale that has the chaotic representation property. Then, from the space of square integrable functionals of $M$,
one can construct generalized functionals of $M$. In this paper, by using a type of weights,
we introduce a class of continuous linear operators acting on generalized functionals of $M$, which we call generalized weighted number (GWN) operators.
We prove that GWN operators can be represented in terms of generalized annihilation and creation operators (acting on generalized functionals of $M$).
We also examine commutation relations between a GWN operator and a generalized annihilation (or creation) operator,
and obtain several formulas expressing such commutation relations.
\vskip 2mm

\vskip 2mm


\section{Introduction}

Hida's white noise analysis is essentially a theory of infinite dimensional calculus on generalized functionals of Brownian motion \cite{hida,huang,obata}.
In 1988, It\^{o} introduced his analysis of Poisson functionals \cite{ito}, which can be viewed as a theory of infinite dimensional
calculus on generalized functionals of Poisson martingale. As is known, both Brownian motion and Poisson martingale are continuous-time normal martingales.
There are theories of infinite dimensional calculus on generalized functionals of some other continuous-time processes in the literature
(see, e.g. \cite{albe,bar,hu,nunno,lee} and references therein).

In recent years, there has been much interest in functionals of discrete-time normal martingales.
Privault \cite{privault} developed his chaotic calculus for functionals of a discrete-time normal martingale.
Nourdin et al. \cite{nourdin} considered a normal approximation of Rademacher functionals (functionals of a special discrete-time normal martingale).
Wang et al. \cite{w-c-l} investigated the annihilation and creation operators acting on square integrable functionals of
a special class of discrete-time normal martingales from a perspective of quantum probability.
Krokowski et al. \cite{krokowski} obtained the Berry-Esseen bounds for multivariate normal approximation of functionals of Rademacher functionals.

Recall that a Gel'fand triple consists of three spaces in the way that $\mathcal{E}\subset \mathcal{H} \subset \mathcal{E}^*$,
where $\mathcal{H}$ is a Hilbert space, $\mathcal{E}$ is a countably-Hilbertian nuclear space that is embedded continuously and densely into $\mathcal{H}$,
and $\mathcal{E}^*$ is the the strong topology dual of $\mathcal{E}$. A Gel'fand triple is said to be complex if the three spaces involved are complex.

Let $M$ be a discrete-time normal martingale having the chaotic representation property and $\mathcal{L}^2(M)$ the complex Hilbert space of
square integrable functionals of $M$. Then, from $\mathcal{L}^2(M)$ and some positive self-adjoint operator $A$ densely-defined in it \cite{wang-chen-1},
one can construct a complex Gel'fand triple
\begin{equation}\label{eq-1-1}
\mathcal{S}(M) \subset \mathcal{L}^2(M) \subset \mathcal{S}^*(M).
\end{equation}
Conventionally, elements of $\mathcal{S}^*(M)$ are called generalized functionals of $M$, while elements of $\mathcal{S}(M)$ are known as testing functionals of $M$.
In 2015, Wang and Chen \cite{wang-chen-1} established a characterization theorem for generalized functionals of $M$,
i.e. elements of $\mathcal{S}^*(M)$, via the Fock transform.
Two years later, they \cite{wang-chen-3} obtained a similar result for continuous linear operators from $\mathcal{S}(M)$ to $\mathcal{S}^*(M)$ via the 2D-Fock transform.

It is known \cite{alicki, partha} that quantum Markov semigroups (QMS), which are quantum analogs of classical Markov semigroups in the theory of probability, play an important role
in describing the irreversible evolution of a quantum system interacting with the environment, i.e., an open quantum system.
Following the ideas of \cite{wang-tang}, one can consider, in the framework of $\mathcal{L}^2(M)$,
the existence of a QMS with a formal generator of the following form
\begin{equation}\label{eq-1-2}
\begin{split}
  \mathcal{L}(X)
    &= \mathrm{i}[H,X]\\
    &\quad  -\frac{1}{2}\sum_{j,k=0}^{\infty}w(j,k)\big[ X \big(\partial_j^*\partial_k\big)^*\partial_j^*\partial_k
           - 2 \big(\partial_j^*\partial_k\big)^* X\partial_j^*\partial_k + \big(\partial_j^*\partial_k\big)^*\partial_j^*\partial_kX\big],
\end{split}
\end{equation}
where $[\cdot,\cdot]$ means the Lie bracket, $\{\partial_k, \partial_k^*\mid k\geq 0\}$ are annihilation and creation operators on $\mathcal{L}^2(M)$
(see Appendix for their definitions), $H$ is a self-adjoint operator densely-defined in $\mathcal{L}^2(M)$, $w$ is some nonnegative function on $\mathbb{N}\times \mathbb{N}$,
and the variable $X$ ranges over all bounded operators on $\mathcal{L}^2(M)$.
From a physical point of view, a QMS with a formal generator of form (\ref{eq-1-2})
belongs to the category of quantum exclusion semigroups, and might serve as a model describing
an open quantum system consisting of an arbitrary number of identical Fermi particles,
where $H$ represents the Hamiltonian of the system and $\partial_j^*\partial_k$ represents the jump of particles from site $k$ to site $j$ with
the jump rate $\sqrt{w(j,k)}$.

To prove the existence of a QMS with a formal generator of form (\ref{eq-1-2}), one needs to deal with the double operator series
\begin{equation}\label{eq-1-3}
  \sum_{j,k=0}^{\infty}w(j,k)\big(\partial_j^*\partial_k\big)^*\partial_j^*\partial_k
  =\sum_{j,k=0}^{\infty}w(j,k)\partial_k^*\partial_j\partial_j^*\partial_k
\end{equation}
and its sum operator in the framework of $\mathcal{L}^2(M)$.
Recently, to deal with a double operator series like (\ref{eq-1-3}) and its sum operator,
Wang et al. \cite{wang-tang} actually  introduced the notion of weighted number operator in $\mathcal{L}^2(M)$ (for short, WN operators in $\mathcal{L}^2(M)$,
or WN operators).
And, also in the same paper, they actually obtained diagonal representations of WN operators and their commutation relations
with $\{\partial_k,\partial_k^*\mid k\geq 0\}$.

In the mathematical descriptions of quantum physics, many observables and quantities are not bona fide Hilbert space operators,
but can only be regarded as generalized operators (e.g., operators on the dual of a nuclear space).
This observation, together with the fact that $\mathcal{S}(M)\subset \mathcal{L}^2(M)\subset \mathcal{S}^*(M)$ is a complex Gel'fand triple,
implies that operators on $\mathcal{S}^*(M)$ can have potential application in quantum physics.
On the other hand, being defined in $\mathcal{L}^2(M)$, WN operators obviously belong to the category of Hilbert space operators.
Hence, it is natural and meaningful to extend the work of \cite{wang-tang} on WN operators to the more general case of $\mathcal{S}^*(M)$.
In this paper, we would like to make such an extension. More specifically, we would like to introduce $\mathcal{S}^*(M)$-analogs of WN operators,
and examine their properties in the framework of $\mathcal{S}^*(M)$.

The paper is organized as follows. In Section~\ref{sec-2}, we describe in detail the Gel'fand triple indicated in (\ref{eq-1-1}) and recall its main properties.
Section~\ref{sec-3} is our main work. Here, by using a type of weight functions, we first introduce a class of continuous linear operators on $\mathcal{S}^*(M)$,
which we call generalized weighted number operators on $\mathcal{S}^*(M)$ (for short, GWN operators on $\mathcal{S}^*(M)$, or  GWN operators).
We then compare GWN operators and WN operators,
and prove a formula that gives a representation of GWN operators in terms of generalized annihilation and creation operators on $\mathcal{S}^*(M)$.
Additionally, in the same section, we examine commutation relations between GWN operators and
generalized annihilation (creation) operators on $\mathcal{S}^*(M)$, and prove several formulas expressing such commutation relations.
Finally in Section~\ref{sec-4}, we provide an appendix about WN operators and their main properties,
which were essentially introduced and proven in \cite{wang-tang}.

Throughout this paper, $\mathbb{N}$ stands  for the set of all nonnegative integers and $\Gamma$
the finite power set of $\mathbb{N}$, namely
\begin{equation}
\Gamma = \{\,\sigma \mid \text{$\sigma \subset \mathbb{N}$ and $\#(\sigma) < \infty$} \,\},
\end{equation}
where $\#(\sigma)$ means the cardinality of $\sigma$ as a set. By convention,
$\mathbb{R}$ denotes the real numbers, while $\mathbb{R}_+$ means the nonnegative real numbers.
Unless otherwise specified, letters like $j$, $k$ and $n$ always mean nonnegative integers.

\section{Preliminaries}\label{sec-2}

In what follows, we always assume that $(\Omega, \mathcal{F}, P)$ is a given probability space.
As usual, we denote by $\mathcal{L}^{2}(\Omega, \mathcal{F}, P)$ the Hilbert space of all
square integrable complex-valued measurable functions on $(\Omega, \mathcal{F}, P)$,
and use $\langle\cdot,\cdot\rangle$ and $\|\cdot\|$ to mean the usual inner product and norm of $\mathcal{L}^{2}(\Omega, \mathcal{F}, P)$, respectively.
By convention, $\langle\cdot,\cdot\rangle$ is conjugate-linear in its first argument and linear in its second argument.

\subsection{Discrete-time normal martingale}

A sequence $M=(M_n)_{n\geq 0}$ of real-valued random variables on $(\Omega, \mathcal{F}, P)$ is called a discrete-time normal martingale
if $\{M_n \mid n\geq 0\} \subset \mathcal{L}^{2}(\Omega, \mathcal{F}, P)$ and
\begin{enumerate}[(i)]
  \item $M=(M_n)_{n\geq 0}$ is a centred martingale with respect to the filtration $(\mathcal{F}_n)_{n\geq 0}$,
        where $\mathcal{F}_n = \sigma(M_k; 0\leq k \leq n)$;
  \item $\mathbb{E}[M_0^2 | \mathcal{F}_{-1}] = 1$ and $\mathbb{E}[M_n^2 | \mathcal{F}_{n-1}] = M_{n-1}^2 +1$
for $n\geq 1$,
\end{enumerate}
where $\mathcal{F}_{-1}=\{\emptyset, \Omega\}$ and $\mathbb{E}[\cdot | \mathcal{F}_{n}]$ means the conditional expectation given $\mathcal{F}_{n}$.

Let $\psi=(\psi_n)_{n\geq 0}$ be a sequence of independent Bernoulli random variables on $(\Omega, \mathcal{F}, P)$ satisfying that $\mathcal{F}=(\psi_n; n \geq 0)$ and
\begin{equation}\label{eq-Bernoulli-sequence}
    P\left\{\psi_n = \sqrt{(1-\theta_n)/\theta_n}\right\}=\theta_n,\quad
    P\left\{\psi_n = -\sqrt{\theta_n/(1-\theta_n)}\right\}=1-\theta_n,\quad n\geq 0,
\end{equation}
where $(\theta_n)_{n\geq 0}$ is a given sequence of positive numbers with the property that $0 < \theta_n < 1$ for all $n\geq 0$.
Then, by putting $M_n=\sum_{k=0}^n \psi_k$, $n\geq 0$, one immediately gets a discrete-time normal martingale $M=(M_n)_{n\geq 0}$,
which we call the discrete-time normal martingale associated with the sequence $\psi=(\psi_n)_{n\geq 0}$.

Let $M=(M_n)_{n\geq0}$ be a discrete-time normal martingale on $(\Omega, \mathcal{F}, P)$.
Then, from $M$, one can construct another sequence $Z=(Z_n)_{n\geq 0}$ of random variables as follows:
\begin{equation}\label{eq-2-1}
Z_0=M_0,\quad Z_n = M_n-M_{n-1},\quad  n\geq 1.
\end{equation}
It can be verified that $Z$ admits the following properties:
\begin{equation}\label{eq-2-2}
    \mathbb{E}[Z_n | \mathcal{F}_{n-1}] =0\quad \text{and}\quad  \mathbb{E}[Z_n^2 | \mathcal{F}_{n-1}] =1,\quad  n\geq 0,
\end{equation}
which means that $Z$ can be viewed as a discrete-time (dependent) noise. In what follows,
we call $Z$ the discrete-time normal noise associated with $M$.

Based on the discrete-time normal noise associated with $M$, one can further construct
a countable system $\{Z_{\sigma}\mid \sigma \in \Gamma\}$ of random variables on $(\Omega, \mathcal{F}, P)$ in the following manner
\begin{equation}\label{eq-2-1}
 Z_{\emptyset}\equiv1\quad \text{and}\quad   Z_{\sigma} = \prod_{j\in \sigma}Z_j,\quad \text{$\sigma \in \Gamma$, $\sigma \neq \emptyset$}.
\end{equation}
Let $\mathcal{F}_{\infty}=\sigma(M_n; n\geq 0)$ be the $\sigma$-field over $\Omega$ generated by $M=(M_n)_{n\geq 0}$
and $\mathcal{L}^2(\Omega, \mathcal{F}_{\infty}, P)$ the space of square integrable complex-valued measurable functions on $(\Omega, \mathcal{F}_{\infty}, P)$,
which is a closed subspace of $\mathcal{L}^2(\Omega, \mathcal{F}, P)$.
If the system $\{Z_{\sigma}\mid \sigma \in \Gamma\}$ is a total subset of $\mathcal{L}^2(\Omega, \mathcal{F}_{\infty}, P)$,
then $M$ is said to have the chaotic representation property.
In that case, the system $\{Z_{\sigma}\mid \sigma \in \Gamma\}$ is actually an orthonormal basis (ONB) for $\mathcal{L}^2(\Omega, \mathcal{F}_{\infty}, P)$.

It is known \cite{privault} that the sequence $\psi=(\psi_n)_{n\geq 0}$ shown above has the chaotic representation property.
Thus, the discrete-time normal martingale associated with $\psi=(\psi_n)_{n\geq 0}$ also has the chaotic representation property.
There are sufficient conditions for a discrete-time normal martingale to have the chaotic representation property
(see, e.g., \cite{emery} and references therein).

\subsection{Generalized functionals of discrete-time normal martingale}\label{subsec-2-2}

From now on, we always assume that $M=(M_n)_{n\geq 0}$ is a fixed discrete-time normal martingale
on $(\Omega, \mathcal{F}, P)$ that has the chaotic representation property and
$Z=(Z_n)_{n\geq 0}$ is the discrete-time normal noise associated with $M$.
For brevity, we set
\begin{equation}\label{eq-SquareIntegrable-Space}
  \mathcal{L}^2(M) \equiv \mathcal{L}^2(\Omega, \mathcal{F}_{\infty}, P)
\end{equation}
and call its elements square integrable functionals of $M$, where $\mathcal{F}_{\infty}=\sigma(M_n; n\geq 0)$ as indicated above.
Note that $\mathcal{L}^2(M)$ shares the same inner product $\langle\cdot,\cdot\rangle$ with $\mathcal{L}^2(\Omega, \mathcal{F}, P)$,
and moreover it has a countable ONB $\{Z_{\sigma}\mid \sigma \in \Gamma\}$,
which is known as the canonical ONB for $\mathcal{L}^2(M)$.

Let $\sigma\mapsto\lambda_{\sigma}$ be the positive integer-valued function on $\Gamma$ defined by
\begin{equation}\label{eq-2-2}
\lambda_{\sigma}=
\left\{
  \begin{array}{ll}
    \prod_{k\in\sigma}(k+1), & \hbox{$\sigma\neq \emptyset$, $\sigma\in\Gamma$;}\\
    1, & \hbox{$\sigma=\emptyset$, $\sigma\in\Gamma$.}
  \end{array}
\right.
\end{equation}
It can be verified that, for any real number $r>1$, the positive term series $\sum_{\sigma\in\Gamma}\lambda^{-r}_{\sigma}$ converges
and moreover its sum satisfies that
\begin{equation}\label{eq-2-3}
\sum_{\sigma\in\Gamma}\lambda^{-r}_{\sigma}\leq \exp\bigg[\sum_{k=1}^{\infty}k^{-r}\bigg]<\infty.
\end{equation}
Denote by $|Z_{\sigma}\rangle\langle Z_{\sigma}|$ the Dirac operator associated with the basis vector $Z_{\sigma}$ for $\sigma\in \Gamma$.
Then one gets a positive self-adjoint operator $A=\sum_{\sigma\in \Gamma}\lambda_{\sigma}|Z_{\sigma}\rangle\langle Z_{\sigma}|$ densely-defined in $\mathcal{L}^2(M)$,
which can be used to construct a Gel'fand triple as follows.

For a nonnegative integer $p\geq 0$, we write $\mathcal{S}_p(M)$ for the domain of the operator $A^p$, namely
\begin{equation}\label{eq-2-4}
  \mathcal{S}_p(M)
  = \mathrm{Dom}\, A^p= \Big\{\, \xi \in \mathcal{L}^2(M) \Bigm| \sum_{\sigma\in \Gamma}\lambda_{\sigma}^{2p}|\langle Z_{\sigma}, \xi\rangle|^{2}< \infty\,\Big\},
\end{equation}
and define a two-variable function $\langle \cdot,\cdot\rangle_p$ on $\mathcal{S}_p(M)$ as
\begin{equation}\label{eq-2-5}
  \langle \xi,\eta\rangle_p
  = \sum_{\sigma\in \Gamma}\lambda_{\sigma}^{2p}\overline{\langle Z_{\sigma},\xi\rangle} \langle Z_{\sigma}, \eta\rangle,\quad
  \xi,\, \eta \in \mathcal{S}_p(M).
\end{equation}
It is not hard to check that, with $\langle \cdot,\cdot\rangle_p$ as its inner product, $\mathcal{S}_p(M)$ becomes a complex Hilbert space.
We write $\|\cdot\|_{p}= \sqrt{\langle \cdot,\cdot\rangle_p}$, which obviously has a representation of the form
\begin{equation}\label{eq-2-6}
  \|\xi\|_{p}^2=\sum_{\sigma\in \Gamma}\lambda_{\sigma}^{2p}|\langle Z_{\sigma}, \xi\rangle|^{2},\quad \xi \in \mathcal{S}_p(M).
\end{equation}
It is can be shown that $\{Z_{\sigma}\mid \sigma\in\Gamma\} \subset \mathcal{S}_p(M)$ and moreover
$\{\lambda^{-p}_{\sigma}Z_{\sigma}\mid \sigma\in\Gamma\}$ forms a countable ONB for $\mathcal{S}_p(M)$.

Due to the fact that $\lambda_{\sigma}\geq 1$ for all $\sigma\in \Gamma$, one has $\mathcal{S}_q(M)\subset \mathcal{S}_p(M)$ and $\|\cdot\|_p \leq \|\cdot\|_q$
whenever $0\leq p \leq q$. Put
\begin{equation}\label{eq-2-8}
  \mathcal{S}(M)=\bigcap_{p\geq 0}\mathcal{S}_{p}(M)
\end{equation}
and endow it with the topology generated by the Hilbertian norm sequence $(\|\cdot \|_{p})_{p\geq 0}$.
Then $\mathcal{S}(M)$ forms a countably Hilbertian nuclear space \cite{becnel, gelfand,wang-chen-1}.
Note that, for each $ p\geq 0$, $\mathcal{S}_p(M)$ is just the completion of $\mathcal{S}(M)$ with respect to $\|\cdot\|_{p}$.

For a nonnegative integer $p\geq 0$, we denote by $\mathcal{S}_p^*(M)$ the dual of $\mathcal{S}_p(M)$ and by $\|\cdot\|_{-p}$
the norm of $\mathcal{S}_p^*(M)$. Then $\mathcal{S}_p^*(M)\subset \mathcal{S}_q^*(M)$ and $\|\cdot\|_{-p} \geq \|\cdot\|_{-q}$ whenever $0\leq p \leq q$.
Let $\mathcal{S}^*(M)$ be the dual of $\mathcal{S}(M)$ and endow it with the strong topology.
Then, as an immediate consequence of the general theory of countably Hilbertian nuclear spaces (see, e.g. \cite{becnel} or \cite{gelfand}),
one has
\begin{equation}\label{eq-2-9}
  \mathcal{S}^*(M)=\bigcup_{p\geq 0}\mathcal{S}_p^*(M),
\end{equation}
and moreover the inductive limit topology over $\mathcal{S}^*(M)$ given by space sequence $\mathcal{S}_p^*(M)$, $p\geq 0$
coincides with the strong topology.

We mention that, by identifying $\mathcal{L}^2(M)$ with its dual, one comes to a Gel'fand triple of the following form
\begin{equation}\label{eq-2-10}
\mathcal{S}(M)\subset \mathcal{L}^2(M)\subset \mathcal{S}^*(M),
\end{equation}
which is referred to as the Gel'fand triple based on functionals of $M$.

\begin{remark}\label{rem-2-2}
Elements of $\mathcal{S}^*(M)$ are known as generalized functionals of $M$, while elements of $\mathcal{S}(M)$ are referred to as testing functionals of $M$.
\end{remark}

Accordingly, $\mathcal{S}^*(M)$ and $\mathcal{S}(M)$ are the generalized functional space and the testing functional space of $M$,
respectively. It turns out \cite{wang-chen-1} that $\mathcal{S}^*(M)$ can accommodate
many quantities of theoretical interest that can not be covered by $\mathcal{L}^2(M)$.
In the following, we denote by $\langle\!\langle \cdot,\cdot\rangle\!\rangle$
the canonical bilinear form on $\mathcal{S}^*(M)\times \mathcal{S}(M)$ given by
\begin{equation}\label{eq-2-12}
  \langle\!\langle \Phi,\xi\rangle\!\rangle = \Phi(\xi),\quad \Phi\in \mathcal{S}^*(M),\, \xi\in \mathcal{S}(M).
\end{equation}
Note that $\langle\!\langle \cdot,\cdot\rangle\!\rangle$ is different from the inner product
$\langle\cdot,\cdot\rangle$ in $\mathcal{L}^2(M)$.
For $\Phi \in \mathcal{S}^*(M)$, its Fock transform is the function $\widehat{\Phi}$ on $\Gamma$ given by
\begin{equation}\label{eq-2-13}
  \widehat{\Phi}(\sigma) = \langle\!\langle \Phi, Z_{\sigma}\rangle\!\rangle,\quad \sigma \in \Gamma.
\end{equation}
It is known \cite{wang-chen-1} that generalized functionals of $M$ are completely determined by their Fock transforms, namely
for $\Phi$, $\Psi \in \mathcal{S}^*(M)$, $\Phi=\Psi$ if and only if $\widehat{\Phi}=\widehat{\Psi}$.
The following lemma comes from \cite{wang-chen-1} (see Theorem 14 and Theorem 15 therein),
which characterizes generalized functionals of $M$ through their Fock transforms.

\begin{lemma}\label{lem-2-4}
Let $F$ be a function on $\Gamma$. Then $F$ is the Fock transform of an element $\Phi$ of $\mathcal{S}^*(M)$ if and only if it satisfies
\begin{equation}\label{eq-2-14}
  |F(\sigma)| \leq C\lambda_{\sigma}^p,\quad \sigma \in \Gamma
\end{equation}
for some constants $C\geq 0$ and $p\geq 0$.
In that case, for $q> p+\frac{1}{2}$, one has
\begin{equation}\label{eq-2-15}
  \|\Phi\|_{-q} \leq C\bigg[\sum_{\sigma \in \Gamma}\lambda_{\sigma}^{-2(q-p)}\bigg]^{\frac{1}{2}}
\end{equation}
and in particular $\Phi \in \mathcal{S}_q^*(M)$.
\end{lemma}

A sequence $(\Phi_n)_{n\geq 1}$ in $\mathcal{S}^*(M)$ is said to converge strongly to an element $\Phi$ of $\mathcal{S}^*(M)$
if converges to $\Phi$ in the strong topology of $\mathcal{S}^*(M)$.
Theorem 10 of \cite{wang-chen-2} provides a criterion for checking whether or not a sequence $(\Phi_n)_{n\geq 0}$ in $\mathcal{S}^*(M)$ converges strongly.

\section{Main work}\label{sec-3}

In this section, we show our main work in the present paper. We continue to use the notions, assumptions and notation
made in previous sections.

\subsection{Definition of generalized weighted number operators}\label{subsec-3-1}

In the first subsection, we first introduce a class of continuous linear operator on $\mathcal{S}^*(M)$, which we call 2D-generalized weighted number (2D-GWN) operators.
And then we compare these operators with the 2D-weighted number (2D-WN) operators introduced in \cite{wang-tang}, which are densely-defined in $\mathcal{L}^2(M)$.
Finally, we prove a regularity property of 2D-GWN operators.

\begin{definition}\label{def-3-1}
A nonnegative function $w$ on $\mathbb{N}\times \mathbb{N}$ is called a 2D-weight if it satisfies that
\begin{equation}\label{eq-weihgt-function}
\sup_{k\geq 0}\sum_{j=0}^{\infty} w(j,k)<\infty.
\end{equation}
\end{definition}

For a 2D-weight $w$, a function $\vartheta_w$ on $\Gamma$, called the spectral function associated with $w$, is defined as
\begin{equation}\label{eq-spectral-function}
\vartheta_w(\sigma)
= \sum_{j=0}^{\infty}\mathbf{1}_{\sigma}(j)w(j,j) + \sum_{j,k=0}^{\infty} \big(1-\mathbf{1}_{\sigma}(j)\big)\mathbf{1}_{\sigma}(k)w(j,k),\quad\sigma \in \Gamma,
\end{equation}
where $\mathbf{1}_{\sigma}(\cdot)$ means the indicator of $\sigma$ as a subset of $\mathbb{N}$.
The next proposition shows that $\vartheta_w$ is dominated by the function $\sigma\mapsto \#(\sigma)$ on $\Gamma$.

\begin{proposition}\label{prop-3-1}
Let $w$ be a 2D-weight and $\vartheta_w$ the spectral function associated with $w$. Then, it holds true that
\begin{equation}\label{eq-3-3}
0\leq \vartheta_w(\sigma)\leq 2\alpha_w \#(\sigma),\quad\sigma \in \Gamma,
\end{equation}
where $\alpha_w=\sup_{k\geq 0}\sum_{j=0}^{\infty} w(j,k)$ and $\#(\sigma)$ means the cardinality of $\sigma$.
\end{proposition}

\begin{proof}
Clearly, $\vartheta_w(\sigma)\geq 0$ for all $\sigma \in \Gamma$. On the other hand, for all $\sigma \in \Gamma$, careful
estimates give
\begin{equation*}
\begin{split}
 \vartheta_w(\sigma)
   &= \sum_{j=0}^{\infty}\mathbf{1}_{\sigma}(j)w(j,j) + \sum_{j,k=0}^{\infty} \big(1-\mathbf{1}_{\sigma}(j)\big)\mathbf{1}_{\sigma}(k)w(j,k)\\
   &\leq \alpha_w\sum_{j=0}^{\infty}\mathbf{1}_{\sigma}(j) + \sum_{k=0}^{\infty}\mathbf{1}_{\sigma}(k)\sum_{j=0}^{\infty}w(j,k)\\
   &\leq 2\alpha_w\#(\sigma).
\end{split}
\end{equation*}
This completes the proof.
\end{proof}

\begin{theorem}\label{basic-theorem-1}
Let $w$ be a 2D-weight. Then, there exists a unique continuous linear operator $\mathfrak{K}_w\colon \mathcal{S}^*(M) \rightarrow \mathcal{S}^*(M)$ such that
\begin{equation}\label{basic-formula-1}
  \widehat{\mathfrak{K}_w\Phi}(\sigma) = \vartheta_w(\sigma)\widehat{\Phi}(\sigma),\quad \sigma\in \Gamma,\; \Phi\in \mathcal{S}^*(M),
\end{equation}
where $\vartheta_w$ is the spectral function associated with $w$.
\end{theorem}

\begin{proof}
Clearly, the operator $\mathfrak{K}_w$ is unique if it exists. Next we show its existence.
Let $\Phi \in \mathcal{S}^*(M)$. By Lemma~\ref{lem-2-4}, there exist constants
$C\geq 0$ and $p\geq 0$ such that
\begin{equation*}
|\widehat{\Phi}(\sigma)|\leq C\lambda_{\sigma}^p,\quad \sigma\in \Gamma,
\end{equation*}
which, together with inequalities $\vartheta_w(\sigma)\leq 2\alpha_w \#(\sigma)$ and $\#(\sigma)\leq \lambda_{\sigma}$,
implies that the function $\sigma\mapsto \vartheta_w(\sigma)\widehat{\Phi}(\sigma)$ satisfies the following relation
\begin{equation*}
  |\vartheta_w(\sigma)\widehat{\Phi}(\sigma)|
\leq 2\alpha_wC\lambda_{\sigma}^{p+1},\quad \sigma\in \Gamma.
\end{equation*}
Thus, again by Lemma~\ref{lem-2-4}, there exists a unique $\Psi_{\Phi}\in \mathcal{S}^*(M)$ such that
\begin{equation*}
\widehat{\Psi_{\Phi}}(\sigma)= \vartheta_w(\sigma)\widehat{\Phi}(\sigma),\quad \sigma \in \Gamma.
\end{equation*}
Now consider the mapping $\mathfrak{K}_w\colon \mathcal{S}^*(M)\rightarrow \mathcal{S}^*(M)$ given by\
$\mathfrak{K}_w \Phi = \Psi_{\Phi}$, $\Phi\in \mathcal{S}^*(M)$.
It is not hard to verify that $\mathfrak{K}_w$ is linear and satisfies (\ref{basic-formula-1}).
Next, we show that $\mathfrak{K}_w$ is continuous.

Let $p\geq 0$ and denote by $\mathfrak{j}_p\colon \mathcal{S}_p^*(M)\rightarrow \mathcal{S}^*(M)$ the inclusion
mapping given by $\mathfrak{j}_p(\Phi)=\Phi$, $\Phi\in \mathcal{S}_p^*(M)$.
Then, for each $\Phi\in \mathcal{S}_p^*(M)$, we have
\begin{equation*}
\begin{split}
\sum_{\sigma\in \Gamma} \lambda_{\sigma}^{-2(p+1)}|\widehat{\mathfrak{K}_w\circ\mathfrak{j}_p(\Phi)}(\sigma)|^2
& = \sum_{\sigma\in \Gamma} \lambda_{\sigma}^{-2(p+1)}|\widehat{\mathfrak{K}_w\Phi}(\sigma)|^2\\
& = \sum_{\sigma\in \Gamma} \lambda_{\sigma}^{-2(p+1)} (\vartheta_w(\sigma))^2|\widehat{\Phi}(\sigma)|^2\\
& \leq 4\alpha_w^2\sum_{\sigma\in \Gamma} \lambda_{\sigma}^{-2p} |\widehat{\Phi}(\sigma)|^2,
\end{split}
\end{equation*}
which implies that $\mathfrak{K}_w\circ\mathfrak{j}_p(\Phi)\in \mathcal{S}_{p+1}^*(M)$ and
$\|\mathfrak{K}_w\circ\mathfrak{j}_p(\Phi)\|_{-(p+1)}\leq 2\alpha_w\|\Phi\|_{-p}$.
Thus $\mathfrak{K}_w\circ\mathfrak{j}_p$ is actually a bounded linear operator from $\mathcal{S}_p^*(M)$ to $\mathcal{S}_{p+1}^*(M)$,
which means that $\mathfrak{K}_w\circ\mathfrak{j}_p\colon \mathcal{S}_p^*(M)\rightarrow \mathcal{S}^*(M)$ is continuous.

Since the choice of the above $p\geq 0$ is arbitrary, we
actually arrive at a conclusion that the composition mapping
$\mathfrak{K}_w\circ\mathfrak{j}_p\colon \mathcal{S}_p^*(M)\rightarrow \mathcal{S}^*(M)$ is continuous
for all $p\geq 0$. Therefore, the mapping $\mathfrak{K}_w\colon \mathcal{S}^*(M)\rightarrow \mathcal{S}^*(M)$ is continuous with respect to
the inductive limit topology over its domain $\mathcal{S}^*(M)$. Since the strong topology over $\mathcal{S}^*(M)$
coincides with the inductive limit topology over $\mathcal{S}^*(M)$, we finally come to the conclusion that $\mathfrak{K}_w$ is continuous.
\end{proof}

\begin{definition}\label{def-2D-GWN-operator}
The operator $\mathfrak{K}_w\colon \mathcal{S}^*(M) \rightarrow \mathcal{S}^*(M)$
indicated in Theorem~\ref{basic-theorem-1} is called the 2D-generalized weighted number operator on $\mathcal{S}^*(M)$ associated with $w$
(for short, the 2D-GWN operator associated with $w$, or the 2D-GWN operator on $\mathcal{S}^*(M)$, or the GWN operator).
\end{definition}

Recall that $\mathcal{L}^2(M)$ is the Hilbert space of square integrable functionals of $M$, which is contained in $\mathcal{S}^*(M)$.
According to Appendix (Section~\ref{sec-4}), for a 2D-weight $w$, the 2D-weighted number (2D-WN) operator $S_w$ associated with $w$
is the one densely-defined in $\mathcal{L}^2(M)$, which can be equivalently redefined as
\begin{equation}\label{eq-3-5}
  S_w\xi = \sum_{\sigma\in \Gamma} \vartheta_w(\sigma)\langle Z_{\sigma},\xi\rangle Z_{\sigma},\quad \xi \in \mathrm{Dom}\,S_w
\end{equation}
with the domain $\mathrm{Dom}\,S_w$ given by
\begin{equation}\label{eq-3-6}
  \mathrm{Dom}\,S_w=\Big\{\xi \in \mathcal{L}^2(M) \Bigm| \sum_{\sigma\in \Gamma} (\vartheta_w(\sigma))^2|\langle Z_{\sigma},\xi\rangle|^2<\infty \Big\}.
\end{equation}
The next proposition shows close links between the 2D-GWN operator $\mathfrak{K}_w$ on $\mathcal{S}^*(M)$ and the 2D-WN operator $S_w$ in $\mathcal{L}^2(M)$,
which justifies Definition~\ref{def-2D-GWN-operator}.

\begin{proposition}\label{prop-3-3}
Let $w$ be a $2$D-weight. Then, the 2D-GWN operator $\mathfrak{K}_w$ on $ \mathcal{S}^*(M)$
and the 2D-WN operator $S_w$ in $\mathcal{L}^2(M)$ admit the following relation
\begin{equation}\label{eq-3-7}
  (\mathsf{R}S_w)\xi = (\mathfrak{K}_w\mathsf{R})\xi,\quad \xi \in \mathrm{Dom}\,S_w,
\end{equation}
where $\mathsf{R}$ is the Riesz mapping from $\mathcal{L}^2(M)$ to its dual.
\begin{equation}\label{eq-3-8}
\begin{CD}
\mathrm{Dom}\,S_w            @> \mathsf{R} >>            \mathcal{S}^*(M)\\
@V S_w VV                                           @VV  \mathfrak{K}_w V \\
\mathcal{L}^2(M)             @>> \mathsf{R} >              \mathcal{S}^*(M)
\end{CD}
\end{equation}
\end{proposition}

\begin{proof}
Let $\xi \in \mathrm{Dom}\,S_w$. Then, for all $\sigma\in \Gamma$, by using (\ref{eq-3-5}) and (\ref{basic-formula-1}), we have
\begin{equation*}
  \widehat{\mathsf{R}(S_w\xi)}(\sigma)=\langle S_w\xi, Z_{\sigma}\rangle = \vartheta_w(\sigma)\langle \xi, Z_{\sigma}\rangle
= \vartheta_w(\sigma)\widehat{\mathsf{R}\xi}(\sigma)= \widehat{\mathfrak{K}_w(\mathsf{R}\xi)}(\sigma),
\end{equation*}
which implies that $\mathsf{R}(S_w\xi)=\mathfrak{K}_w(\mathsf{R}\xi)$, namely $(\mathsf{R}S_w)\xi=(\mathfrak{K}_w\mathsf{R})\xi$.
\end{proof}

In general, the 2D-WN operator $S_w$ is not bounded (hence not continuous) and its domain $\mathrm{Dom}\,S_w$ is a proper subspace of $\mathcal{L}^2(M)$
(see Lemma~\ref{lem-4-1} and Remark~\ref{rem-4-1} in Appendix for details).
However, as an $\mathcal{S}^*(M)$-version of $S_w$, the 2D-GWN operator $\mathfrak{K}_w$ is defined on the whole space $\mathcal{S}^*(M)$ and is always continuous,
which might make it convenient to apply $\mathfrak{K}_w$ in many cases.
The next theorem further shows the regularity of $\mathfrak{K}_w$.

\begin{theorem}\label{thr-3-3}
Let $w$ be a 2D-weight. Then, for each $p\geq 0$, one has $\mathfrak{K}_w(\mathcal{S}_p^*(M))\subset \mathcal{S}_{p+1}^*(M)$,
and moreover it holds that
\begin{equation}\label{eq}
  \|\mathfrak{K}_w\Phi\|_{-(p+1)}\leq 2\alpha_w\|\Phi\|_{-p},\quad \Phi\in \mathcal{S}_p^*(M),
\end{equation}
where $\alpha_w=\sup_{k\geq 0}\sum_{j=0}^{\infty} w(j,k)$.
\end{theorem}

\begin{proof}
Let $\Phi\in \mathcal{S}_p^*(M)$ be given. Then, by using Theorem~\ref{basic-theorem-1} and Proposition~\ref{prop-3-1},  we find
\begin{equation*}
 \sum_{\sigma\in \Gamma} \lambda_{\sigma}^{-2(p+1)}|\widehat{\mathfrak{K}_w\Phi}(\sigma)|^2
 = \sum_{\sigma\in \Gamma} \lambda_{\sigma}^{-2(p+1)} (\vartheta_w(\sigma))^2|\widehat{\Phi}(\sigma)|^2
 \leq 4\alpha_w^2\sum_{\sigma\in \Gamma} \lambda_{\sigma}^{-2p} |\widehat{\Phi}(\sigma)|^2,
\end{equation*}
which together Lemma~\ref{lem-2-4} implies that $\mathfrak{K}_w\Phi\in \mathcal{S}_{p+1}^*(M)$ and
$\|\mathfrak{K}_w\Phi\|_{-(p+1)}\leq 2\alpha_w\|\Phi\|_{-p}$.
\end{proof}

\subsection{Representation of generalized weighted number operators}\label{subsec-3-2}

As is seen in Subsection~\ref{subsec-3-1}, 2D-GWN operators are actually a type of ``diagonal operators'' on $\mathcal{S}^*(M)$.
In the present subsection, we first show that 2D-GWN operators can be expressed in terms of generalized annihilation and creation operators on $\mathcal{S}^*(M)$.
And then we prove several further results about a special class of 2D-GWN operators, which we call 1D-GWN operators.

First, we recall the definition of generalized annihilation and creation operators on $\mathcal{S}^*(M)$ and their basic properties
(Cf Subsection 3.1 of \cite{wang-Lin}).
For $k\geq 0$, the generalized annihilation operator associated with $k$ is the continuous linear operator
$\mathfrak{a}_k$ on $\mathcal{S}^*(M)$ determined by
\begin{equation}\label{eq-annihilation}
\widehat{\mathfrak{a}_k\Phi}(\sigma)=[1-\mathbf{1}_{\sigma}(k)]\widehat{\Phi}(\sigma\cup k),\quad \sigma\in\Gamma,\,\Phi\in\mathcal{S}^*(M)
\end{equation}
and the generalized creation operators associated with $k$ is the continuous linear operator
$\mathfrak{a}_k^{\dag}$ on $\mathcal{S}^*(M)$ determined by
\begin{equation}\label{eq-creation}
\widehat{\mathfrak{a}_k^{\dag}\Phi}(\sigma)=\mathbf{1}_{\sigma}(k)\widehat{\Phi}(\sigma\setminus k),\quad \sigma\in\Gamma,\,\Phi\in\mathcal{S}^*(M),
\end{equation}
where $\sigma\cup k=\sigma\cup \{k\}$, $\sigma\setminus k=\sigma\setminus \{k\}$
and $\widehat{\Psi}$ means the Fock transform of a generalized functional $\Psi\in \mathcal{S}^*(M)$.

Let $\partial_k$ and $\partial_k^*$ be the annihilation and creation operators on $\mathcal{L}^2(M)$ associated with $k\geq 0$ (see Appendix for details).
Then it can be shown that
\begin{equation}\label{eq-3-12}
  \mathsf{R}\partial_k= \mathfrak{a}_k\mathsf{R},\quad \mathsf{R}\partial_k^*= \mathfrak{a}_k^{\dag}\mathsf{R},
\end{equation}
where $\mathsf{R}$ denotes the Riesz mapping from $\mathcal{L}^2(M)$ to its dual.
The above relations suggest that $\mathfrak{a}_k$ and $\mathfrak{a}_k^{\dag}$ can be viewed as $\mathcal{S}^*(M)$-versions of $\partial_k$ and $\partial_k^*$, respectively.
However, it should be emphasized that $\partial_k$ and $\partial_k^*$ are bounded (continuous) linear operators on $\mathcal{L}^2(M)$,
while $\mathfrak{a}_k$ and $\mathfrak{a}_k^{\dag}$ are continuous linear operators on $\mathcal{S}^*(M)$.
Additionally, $\partial_k^*$ is the adjoint of $\partial_k$, while $\mathfrak{a}_k^{\dag}$ is not the adjoint of $\mathfrak{a}_k$.
In fact, the adjoint of $\mathfrak{a}_k$ is a continuous linear operator on the testing functional space $\mathcal{S}(M)$.

Just like annihilation and creation operators on $\mathcal{L}^2(M)$, generalized annihilation and creation operators on $\mathcal{S}^*(M)$
also satisfy a canonical anti-commutation relation (CAR) in equal time, namely it holds true that
\begin{equation}\label{eq-3-13}
\mathfrak{a}_k^{\dag}\mathfrak{a}_k + \mathfrak{a}_k\mathfrak{a}_k^{\dag}=I,\quad k\geq 0,
\end{equation}
where $I$ means the identity operator on $\mathcal{S}^*(M)$. Additionally, one can also verify that
\begin{equation}\label{eq-3-14}
 \mathfrak{a}_j\mathfrak{a}_k = \mathfrak{a}_k\mathfrak{a}_j,\quad
 \mathfrak{a}_j^{\dag}\mathfrak{a}_k^{\dag} = \mathfrak{a}_k^{\dag}\mathfrak{a}_j^{\dag},\quad
\mathfrak{a}_j^{\dag}\mathfrak{a}_k = \mathfrak{a}_k\mathfrak{a}_j^{\dag}\qquad (j\neq k)
\end{equation}
and
\begin{equation}\label{eq-3-15}
\mathfrak{a}_j\mathfrak{a}_j = 0,\quad \mathfrak{a}_j^{\dag}\mathfrak{a}_j^{\dag}=0,
\end{equation}
where $j$, $k\geq 0$ and the righthand sides of (\ref{eq-3-15}) mean the null operator on $\mathcal{S}^*(M)$.

We now consider series in $\mathcal{S}^*(M)$. A double series $\sum_{j,k=0}^{\infty}\Phi_{jk}$ in $\mathcal{S}^*(M)$ is said to converge naturally
if its sequence of square-array partial sums
\begin{equation*}
  \Phi^{(n)}=\sum_{j,k=0}^n\Phi_{jk},\quad n\geq 0
\end{equation*}
converges strongly in $\mathcal{S}^*(M)$ as $n\rightarrow \infty$. In that case, we write $\sum_{j,k=0}^{\infty}\Phi_{jk}=\lim_{n\to \infty}\Phi^{(n)}$.

\begin{proposition}\label{prop-3-5}
Let $\sigma\in \Gamma$ and $\Phi\in \mathcal{S}^*(M)$. Then, for $j$, $k\geq 0$, one has
\begin{equation}\label{eq}
  \widehat{\mathfrak{a}_k^{\dag}\mathfrak{a}_j\mathfrak{a}_j^{\dag}\mathfrak{a}_k\Phi}(\sigma)=
  \left\{
    \begin{array}{ll}
      (1-\mathbf{1}_{\sigma}(j))\mathbf{1}_{\sigma}(k)\widehat{\Phi}(\sigma), & \hbox{$j\neq k$;} \\
      \mathbf{1}_{\sigma}(j)\widehat{\Phi}(\sigma), & \hbox{$j=k$,}
    \end{array}
  \right.
\end{equation}
where $\widehat{\mathfrak{a}_k^{\dag}\mathfrak{a}_j\mathfrak{a}_j^{\dag}\mathfrak{a}_k\Phi}$ denotes
the Fock transform of $\mathfrak{a}_k^{\dag}\mathfrak{a}_j\mathfrak{a}_j^{\dag}\mathfrak{a}_k\Phi$.
In particular, $\mathfrak{a}_j^{\dag}\mathfrak{a}_j\mathfrak{a}_j^{\dag}\mathfrak{a}_j = \mathfrak{a}_j^{\dag}\mathfrak{a}_j$ for all $j\geq 0$.
\end{proposition}

\begin{proof}
When $j\neq k$, we have $\mathbf{1}_{\sigma\setminus k}(j) = \mathbf{1}_{\sigma}(j)$,
which together with (\ref{eq-annihilation}) and (\ref{eq-creation}) gives
\begin{equation*}
\begin{split}
  \widehat{\mathfrak{a}_k^{\dag}\mathfrak{a}_j\mathfrak{a}_j^{\dag}\mathfrak{a}_k\Phi}(\sigma)
&= \mathbf{1}_{\sigma}(k)(1-\mathbf{1}_{\sigma}(j))\widehat{\mathfrak{a}_j^{\dag}\mathfrak{a}_k\Phi} ((\sigma\setminus k)\cup j)\\
&= \mathbf{1}_{\sigma}(k)(1-\mathbf{1}_{\sigma}(j))\widehat{\mathfrak{a}_k\Phi}(\sigma\setminus k)\\
&=\mathbf{1}_{\sigma}(k)(1-\mathbf{1}_{\sigma}(j))\widehat{\Phi}(\sigma).
\end{split}
\end{equation*}
When $j= k$, we can similarly get
$\widehat{\mathfrak{a}_k^{\dag}\mathfrak{a}_j\mathfrak{a}_j^{\dag}\mathfrak{a}_k\Phi}(\sigma)
=\widehat{\mathfrak{a}_j^{\dag}\mathfrak{a}_j\mathfrak{a}_j^{\dag}\mathfrak{a}_j\Phi}(\sigma)
=\mathbf{1}_{\sigma}(j)\widehat{\Phi}(\sigma)$.
Finally, a simple calculation shows that $\widehat{\mathfrak{a}_j^{\dag}\mathfrak{a}_j\Phi}(\sigma) =\mathbf{1}_{\sigma}(j)\widehat{\Phi}(\sigma)$,
thus
\begin{equation*}
  \widehat{\mathfrak{a}_j^{\dag}\mathfrak{a}_j\mathfrak{a}_j^{\dag}\mathfrak{a}_j\Phi}(\sigma)
= \widehat{\mathfrak{a}_j^{\dag}\mathfrak{a}_j\Phi}(\sigma),
\end{equation*}
which together with the arbitrariness of $\sigma\in \Gamma$ and $\Phi\in \mathcal{S}^*(M)$ implies that
$\mathfrak{a}_j^{\dag}\mathfrak{a}_j\mathfrak{a}_j^{\dag}\mathfrak{a}_j = \mathfrak{a}_j^{\dag}\mathfrak{a}_j$.
\end{proof}

The following theorem gives a formula that expresses a 2D-GWN operator in terms of generalized annihilation and creation operators.

\begin{theorem}\label{eq-2D-weighted-operator-representation}
Let $w$ be a 2D-weight and $\mathfrak{K}_w$ the 2D-GWN operator. Then, it holds true that
\begin{equation}\label{eq-GWO-representation}
  \mathfrak{K}_w\Phi = \sum_{j,k=0}^{\infty}w(j,k)\mathfrak{a}_k^{\dag}\mathfrak{a}_j\mathfrak{a}_j^{\dag}\mathfrak{a}_k\Phi,\quad \Phi \in \mathcal{S}^*(M),
\end{equation}
where the double series on the righthand side converges naturally.
\end{theorem}

\begin{proof}
Let $\Phi\in \mathcal{S}^*(M)$. We need to show that the series on the righthand side converges naturally to $\mathfrak{K}_w\Phi$.
To this end, we write
\begin{equation*}
  \Psi_n = \sum_{j,k=0}^nw(j,k)\mathfrak{a}_k^{\dag}\mathfrak{a}_j\mathfrak{a}_j^{\dag}\mathfrak{a}_k\Phi,\quad n\geq 0.
\end{equation*}
By Proposition~\ref{prop-3-5}, we have
\begin{equation*}
\begin{split}
  \widehat{\Psi_n}(\sigma)
   & = \sum_{j,k=0}^nw(j,k)\widehat{\mathfrak{a}_k^{\dag}\mathfrak{a}_j\mathfrak{a}_j^{\dag}\mathfrak{a}_k\Phi}(\sigma)\\
   & = \sum_{j=0}^n\Big[\mathbf{1}_{\sigma}(j)w(j,j) +
        \sum_{k=0,k\neq j}^n(1-\mathbf{1}_{\sigma}(j))\mathbf{1}_{\sigma}(k)w(j,k) \Big]\widehat{\Phi}(\sigma)\\
   & = \Big[\sum_{j=0}^n\mathbf{1}_{\sigma}(j)w(j,j) +
        \sum_{j,k=0}^n(1-\mathbf{1}_{\sigma}(j))\mathbf{1}_{\sigma}(k)w(j,k) \Big]\widehat{\Phi}(\sigma),
\end{split}
\end{equation*}
where $\sigma\in \Gamma$. This together with Definition~\ref{def-3-1} and Theorem~\ref{basic-theorem-1} implies that
\begin{equation*}
 \lim_{n\to \infty} \widehat{\Psi_n}(\sigma)= \vartheta_w(\sigma)\widehat{\Phi}(\sigma)=\widehat{\mathfrak{K}_w\Phi}(\sigma),\quad
 \forall\,\sigma\in \Gamma.
\end{equation*}
On the other hand, by the characterization theorem of generalized functionals (Lemma~\ref{lem-2-4}), there exist constants $C\geq 0$ and
$p\geq 0$ such that
\begin{equation*}
  |\widehat{\Phi}(\sigma)|\leq C\lambda_{\sigma}^p,\quad \sigma\in \Gamma,
\end{equation*}
which, together with inequalities $\vartheta_w(\sigma)\leq 2\alpha_w \#(\sigma)$
and $\#(\sigma)\leq \lambda_{\sigma}$, yields
\begin{equation*}
\begin{split}
  \sup_{n\geq 0}|\widehat{\Psi_n}(\sigma)|
    &= \sup_{n\geq 0}\Big[\sum_{j=0}^n\mathbf{1}_{\sigma}(j)w(j,j) +
        \sum_{j,k=0}^n(1-\mathbf{1}_{\sigma}(j))\mathbf{1}_{\sigma}(k)w(j,k) \Big]|\widehat{\Phi}(\sigma)|\\
    & \leq \vartheta_w(\sigma)|\widehat{\Phi}(\sigma)|\\
    & \leq 2\alpha_wC\lambda_{\sigma}^{p+1},
\end{split}
\end{equation*}
where $\sigma \in \Gamma$.
Therefore, using the criterion for convergence of generalized functional sequences (see Theorem 10 of \cite{wang-chen-2}), we come to the conclusion that
$\Psi_n \rightarrow \mathfrak{K}_w\Phi$ strongly as $n\rightarrow \infty$, namely (\ref{eq-GWO-representation}) holds true.
\end{proof}

\begin{remark}\label{rem-3-1}
Let $w$ be a 2D-weight. Then, according to Appendix (Section~\ref{sec-4}), the 2D-WN operator $S_w$ in $\mathcal{L}^2(M)$ has a
representation of the form
\begin{equation}\label{eq-3-18}
  S_w\xi=\sum_{j,k=0}^{\infty}w(j,k)\partial_k^*\partial_j\partial_j^*\partial_k\xi,\quad \xi \in \mathrm{Dom}\, S_w.
\end{equation}
Comparing (\ref{eq-GWO-representation}) and (\ref{eq-3-18}), one can further see the similarity as well as the difference between
the 2D-GWN operator $\mathfrak{K}_w$ on $\mathcal{S}^*(M)$ and the 2D-WN operator $S_w$ in $\mathcal{L}^2(M)$.
\end{remark}

We now turn our attention to a special class of 2D-GWN operators, which are actually determined by one-variable functions on $\mathbb{N}$.

\begin{theorem}\label{thr-1D-weighted-number-operator}
Let $u$ be a bounded nonnegative function on $\mathbb{N}$. Then there exists a unique continuous linear operator
$\mathfrak{N}_u \colon \mathcal{S}^*(M)\rightarrow \mathcal{S}^*(M)$ such that
\begin{equation}\label{eq-2D-weight-induced-by-1D}
  \widehat{\mathfrak{N}_u\Phi}(\sigma) = \#_u(\sigma)\widehat{\Phi}(\sigma),\quad \sigma\in \Gamma,\, \Phi \in \mathcal{S}^*(M),
\end{equation}
where $\#_u(\sigma)=\sum_{k=0}^{\infty}\mathbf{1}_{\sigma}(k)u(k)$.
\end{theorem}

\begin{proof}
Define a nonnegative function $w^{(u)}$ on $\mathbb{N}\times \mathbb{N}$ as follows
\begin{equation*}
  w^{(u)}(j,k)=
  \left\{
    \begin{array}{ll}
      u(k), & \hbox{$j=k$, $(j,k)\in \mathbb{N}\times \mathbb{N}$;} \\
      0, & \hbox{$j\neq k$, $(j,k)\in \mathbb{N}\times \mathbb{N}$.}
    \end{array}
  \right.
\end{equation*}
It is easy to verify that $w^{(u)}$ is a $2$D-weight and $\sup_{k\geq 0}\sum_{j=0}^{\infty}w^{(u)}(j,k)=\sup_{k\geq 0}u(k)$.
And moreover the spectral function $\vartheta_{w^{(u)}}$ associated with $w^{(u)}$ coincides with the function $\#_u(\cdot)$, namely
\begin{equation*}
  \vartheta_{w^{(u)}}(\sigma) = \#_u(\sigma),\quad \sigma\in \Gamma.
\end{equation*}
Now put $\mathfrak{N}_u=\mathfrak{K}_{w^{(u)}}$. Then $\mathfrak{N}_u$ is the desired operator.
\end{proof}

\begin{definition}\label{def-1D-weighted-number-operator}
For a bounded nonnegative function $u$ on $\mathbb{N}$, the operator $\mathfrak{N}_u$ indicated in Theorem~\ref{thr-1D-weighted-number-operator}
is called the 1D-generalized weighted number operator (for short, the 1D-GWN operator) on $\mathcal{S}^*(M)$ associated with $u$.
\end{definition}

As can be seen in the proof of Theorem~\ref{thr-1D-weighted-number-operator}, 1D-GWN operators are actually a special class of 2D GWN operator on $\mathcal{S}^*(M)$.
Hence they can be expected to have better properties. This is indeed the case.

\begin{theorem}\label{thr-GWO-regularity-1D}
Let $\mathfrak{N}_u$ be the 1D-GWN operator on $\mathcal{S}^*(M)$ associated with a bounded nonnegative function $u$ on $\mathbb{N}$.
Then, for each $p\geq 0$, one has $\mathfrak{N}_u(\mathcal{S}_p^*(M))\subset \mathcal{S}_{p+1}^*(M)$ and moreover
\begin{equation}\label{eq-1D-norm-estimate}
  \|\mathfrak{N}_u\Phi\|_{-(p+1)}\leq \beta_u\|\Phi\|_{-p},\quad \Phi \in \mathcal{S}_p^*(M).
\end{equation}
where $\beta_u=\sup_{k\geq 0}u(k)$.
\end{theorem}

\begin{proof}
Let $\Phi \in \mathcal{S}_p^*(M)$. By using Theorem~\ref{thr-1D-weighted-number-operator}, we find
\begin{equation*}
  \sum_{\sigma\in \Gamma}\lambda_{\sigma}^{-2(p+1)}|\widehat{\mathfrak{N}_u\Phi}(\sigma)|^2
  = \sum_{\sigma\in \Gamma}\lambda_{\sigma}^{-2(p+1)}(\#_u(\sigma))^2|\widehat{\Phi}(\sigma)|^2
 \leq \beta_u^2\sum_{\sigma\in \Gamma}\lambda_{\sigma}^{-2(p+1)}(\#(\sigma))^2|\widehat{\Phi}(\sigma)|^2,
\end{equation*}
which together with $\#(\sigma)\leq \lambda_{\sigma}$ implies that
\begin{equation*}
  \sum_{\sigma\in \Gamma}\lambda_{\sigma}^{-2(p+1)}|\widehat{\mathfrak{N}_u\Phi}(\sigma)|^2
  \leq \beta_u^2\sum_{\sigma\in \Gamma}\lambda_{\sigma}^{-2p}|\widehat{\Phi}(\sigma)|^2,
\end{equation*}
which means that $\mathfrak{N}_u\Phi\in \mathcal{S}_{p+1}^*(M)$ and $\|\mathfrak{N}_u\Phi\|_{-(p+1)}\leq \beta_u\|\Phi\|_{-p}$.
\end{proof}

\begin{remark}\label{rem-3-2}
By using Theorem~\ref{thr-3-3}, one can only get $\|\mathfrak{N}_u\Phi\|_{-(p+1)}\leq 2\beta_u\|\Phi\|_{-p}$, which is not as good as
(\ref{eq-1D-norm-estimate}).
\end{remark}

\begin{theorem}\label{thr-3-9}
Let $\mathfrak{N}_u$ be the 1D-GWN operator associated with a bounded nonnegative function $u$ on $\mathbb{N}$.
Then, $\mathfrak{N}_u$ has a representation of the following form
\begin{equation}\label{eq}
  \mathfrak{N}_u\Phi =\sum_{k=0}^{\infty}u(k)\mathfrak{a}_k^{\dag}\mathfrak{a}_k\Phi,\quad \Phi \in \mathcal{S}^*(M),
\end{equation}
where the series on the righthand side converges strongly.
\end{theorem}

\begin{proof}
Let $\Phi \in \mathcal{S}^*(M)$. According to Proposition~\ref{prop-3-5} and the proof of Theorem~\ref{thr-1D-weighted-number-operator}, we have
\begin{equation*}
\sum_{k=0}^nu(k)\mathfrak{a}_k^{\dag}\mathfrak{a}_k\Phi
= \sum_{k=0}^nu(k)\mathfrak{a}_k^{\dag}\mathfrak{a}_k\mathfrak{a}_k^{\dag}\mathfrak{a}_k\Phi
= \sum_{j,k=0}^nw^{(u)}(j,k)\mathfrak{a}_k^{\dag}\mathfrak{a}_j\mathfrak{a}_j^{\dag}\mathfrak{a}_k\Phi,\quad n\geq 1,
\end{equation*}
Thus, by Theorem~\ref{eq-2D-weighted-operator-representation} and the proof of Theorem~\ref{thr-1D-weighted-number-operator}, we get to know that
\begin{equation*}
   \sum_{k=0}^{\infty}u(k)\mathfrak{a}_k^{\dag}\mathfrak{a}_k\Phi
     = \sum_{j,k=0}^{\infty}w^{(u)}(j,k)\mathfrak{a}_k^{\dag}\mathfrak{a}_j\mathfrak{a}_j^{\dag}\mathfrak{a}_k\Phi
     = \mathfrak{K}_{w^{(u)}}\Phi= \mathfrak{N}_u\Phi,
\end{equation*}
where the first series converges strongly because the second one converges naturally.
\end{proof}

\begin{corollary}\label{coro-3-10}
There exists a unique continuous linear operator $\mathfrak{N}\colon \mathcal{S}^*(M) \rightarrow \mathcal{S}^*(M)$ such that
\begin{equation}\label{eq}
  \widehat{\mathfrak{N}\Phi}(\sigma) = \#(\sigma)\widehat{\Phi}(\sigma),\quad \sigma\in \Gamma,\, \Phi \in \mathcal{S}^*(M),
\end{equation}
where $\#(\sigma)$ denotes the cardinality of $\sigma$ as a set.
\end{corollary}

\begin{proof}
Let $\mathfrak{N} = \mathfrak{N}_{u_0}$, where $u_0$ is the constant function on $\mathbb{N}$ given by $u_0(j)=1$, $\forall\, j\in \mathbb{N}$.
Then, by Theorem~\ref{thr-1D-weighted-number-operator}, we know that $\mathfrak{N}$ is the desired.
\end{proof}

\begin{remark}\label{rem-3-3}
We call $\mathfrak{N}$ the generalized number operator on $\mathcal{S}^*(M)$.
\end{remark}

From Theorem~\ref{thr-3-9} and the proof of Corollary~\ref{coro-3-10}, one can get a representation of  $\mathfrak{N}$
in terms of the family $\{\mathfrak{a}_k, \mathfrak{a}_k^{\dag}\mid k\geq 0\}$, namely
\begin{equation}\label{eq}
  \mathfrak{N}\Phi =\sum_{k=0}^{\infty}\mathfrak{a}_k^{\dag}\mathfrak{a}_k\Phi,\quad \Phi \in \mathcal{S}^*(M),
\end{equation}
where the series on the righthand side converges strongly.

\subsection{Commutation relations}

In this subsection, we examine commutation relations between GWN operators and generalized annihilation (creation) operators.

\begin{theorem}\label{thr-1D-commutation}
Let $\mathfrak{N}_u$ be the 1D-GWN operator associated with a bounded nonnegative function $u$ on $\mathbb{N}$.
Then, for all $k\geq 0$, it holds true that
\begin{equation}\label{eq}
  \mathfrak{N}_u\mathfrak{a}_k=\mathfrak{a}_k\mathfrak{N}_u -u(k)\mathfrak{a}_k,\quad
   \mathfrak{N}_u\mathfrak{a}_k^{\dag}=\mathfrak{a}_k^{\dag}\mathfrak{N}_u +u(k)\mathfrak{a}_k^{\dag},\quad
   \mathfrak{N}_u\mathfrak{a}_k^{\dag}\mathfrak{a}_k= \mathfrak{a}_k^{\dag}\mathfrak{a}_k\mathfrak{N}_u.
\end{equation}
\end{theorem}

\begin{proof}
Let $k\geq 0$. For all $\Phi \in \mathcal{S}^*(M)$ and $\sigma\in \Gamma$, by using properties of $\mathfrak{N}_u$ and $\mathfrak{a}_k$, we have
\begin{equation*}
  \widehat{(\mathfrak{a}_k\mathfrak{N}_u)\Phi}(\sigma)
   = \widehat{\mathfrak{a}_k(\mathfrak{N}_u\Phi)}(\sigma)
   = (1-\mathbf{1}_{\sigma}(k))\widehat{\mathfrak{N}_u\Phi}(\sigma\cup k)
   = (1-\mathbf{1}_{\sigma}(k))\#_u(\sigma\cup k)\widehat{\Phi}(\sigma\cup k),
\end{equation*}
which together with $(1-\mathbf{1}_{\sigma}(k))\#_u(\sigma\cup k) = (1-\mathbf{1}_{\sigma}(k))\#_u(\sigma) +  (1-\mathbf{1}_{\sigma}(k))u(k)$ gives
\begin{equation*}
\begin{split}
  \widehat{(\mathfrak{a}_k\mathfrak{N}_u)\Phi}(\sigma)
   &= (1-\mathbf{1}_{\sigma}(k))\#_u(\sigma)\widehat{\Phi}(\sigma\cup k) + (1-\mathbf{1}_{\sigma}(k))u(k)\widehat{\Phi}(\sigma\cup k)\\
   &= \#_u(\sigma)\widehat{\mathfrak{a}_k\Phi}(\sigma) + u(k)\widehat{\mathfrak{a}_k\Phi}(\sigma)\\
   &= \widehat{\mathfrak{N}_u(\mathfrak{a}_k\Phi)}(\sigma)+ u(k)\widehat{\mathfrak{a}_k\Phi}(\sigma)\\
   &= \widehat{(\mathfrak{N}_u\mathfrak{a}_k)\Phi}(\sigma)+ u(k)\widehat{\mathfrak{a}_k\Phi}(\sigma).
\end{split}
\end{equation*}
Thus $\mathfrak{a}_k\mathfrak{N}_u = \mathfrak{N}_u\mathfrak{a}_k + u(k)\mathfrak{a}_k$, which is equivalent to
$\mathfrak{N}_u\mathfrak{a}_k=\mathfrak{a}_k\mathfrak{N}_u -u(k)\mathfrak{a}_k$. Similarly, we can verify the second equality.
Finally, by using the first and second equalities, we get
\begin{equation*}
  \mathfrak{N}_u\mathfrak{a}_k^{\dag}\mathfrak{a}_k
   =\mathfrak{a}_k^{\dag}\mathfrak{N}_u\mathfrak{a}_k + u(k)\mathfrak{a}_k^{\dag}\mathfrak{a}_k
  = \mathfrak{a}_k^{\dag}(\mathfrak{a}_k\mathfrak{N}_u -u(k)\mathfrak{a}_k) + u(k)\mathfrak{a}_k^{\dag}\mathfrak{a}_k
  = \mathfrak{a}_k^{\dag}\mathfrak{a}_k\mathfrak{N}_u.
\end{equation*}
This completes the proof.
\end{proof}

Comparing Theorem~\ref{thr-1D-commutation} and Lemma~\ref{lem-4-3}, one can see that the commutation relations between $\mathfrak{N}_u$
and $\mathfrak{a}_k$ (respectively, $\mathfrak{a}_k^{\dag}$) are quite similar to those between $N_u$ and $\partial_k$ (respectively, $\partial_k^*$).
However, the commutation relations between $\mathfrak{N}_u$ and $\mathfrak{a}_k$ (respectively, $\mathfrak{a}_k^{\dag}$) hold true on the whole space $\mathcal{S}^*(M)$,
while the commutation relations between between $N_u$ and $\partial_k$ (respectively, $\partial_k^*$) hold true only on $\mathrm{Dom}\, N_u$,
which is a dense subspace of $\mathcal{L}^2(M)$.

\begin{corollary}
The generalized number operator $\mathfrak{N}$ admits the following commutation relations
\begin{equation}\label{eq}
  \mathfrak{N}\,\mathfrak{a}_k=\mathfrak{a}_k\mathfrak{N} -\mathfrak{a}_k,\quad
   \mathfrak{N}\,\mathfrak{a}_k^{\dag}=\mathfrak{a}_k^{\dag}\mathfrak{N} + \mathfrak{a}_k^{\dag},\quad
   \mathfrak{N}\,\mathfrak{a}_k^{\dag}\mathfrak{a}_k= \mathfrak{a}_k^{\dag}\mathfrak{a}_k\mathfrak{N}.
\end{equation}
\end{corollary}

As is seen, there are relatively simple commutation relations between 1D-GWN operators and
generalized annihilation (creation) operators. Next, we turn our attention to the case of 2D-GWN operators.

\begin{proposition}\label{prop-3-12}
Let $w$ be a 2D-weight. Then, for all $k\geq 0$ and $\sigma\in \Gamma$, it holds true that
\begin{equation}\label{eq}
  (1-\mathbf{1}_{\sigma}(k))\vartheta_w(\sigma\cup k)
   = (1-\mathbf{1}_{\sigma}(k))\Big[\vartheta_w(\sigma)-\#_{w(k,\cdot)}(\sigma)-\#_{w(\cdot,k)}(\sigma) + \sum_{j=0}^{\infty}w(j,k)\Big],
\end{equation}
where $\#_{w(k,\cdot)}(\sigma)= \sum_{n=0}^{\infty}\mathbf{1}_{\sigma}(n)w(k,n)$ and $\#_{w(\cdot,k)}(\sigma)= \sum_{j=0}^{\infty}\mathbf{1}_{\sigma}(j)w(j,k)$.
\end{proposition}

\begin{proof}
Let $k\geq 0$ and $\sigma\in \Gamma$. Then, by careful calculations, we find
\begin{equation*}
  (1-\mathbf{1}_{\sigma}(k))\sum_{j=0}^{\infty} \mathbf{1}_{\sigma\cup k}(j)w(j,j)
= (1-\mathbf{1}_{\sigma}(k))\Big[\sum_{j=0}^{\infty}\mathbf{1}_{\sigma}(j)w(j,j) + w(k,k)\Big]
\end{equation*}
and
\begin{equation*}
\begin{split}
 (1-&\mathbf{1}_{\sigma}(k))\sum_{j,n=0}^{\infty}(1-\mathbf{1}_{\sigma\cup k}(j))\mathbf{1}_{\sigma\cup k}(n)w(j,n)\\
  & = (1-\mathbf{1}_{\sigma}(k))\Big[\sum_{j,n=0}^{\infty} (1-\mathbf{1}_{\sigma}(j))\mathbf{1}_{\sigma}(n)w(j,n)
        - \sum_{n=0}^{\infty}\mathbf{1}_{\sigma}(n)w(k,n)\\
 &\qquad  +\sum_{j=0}^{\infty}w(j,k) - \sum_{j=0}^{\infty}\mathbf{1}_{\sigma}(j)w(j,k) - w(k,k)\Big]\\
 & = (1-\mathbf{1}_{\sigma}(k))\Big[\sum_{j,n=0}^{\infty} (1-\mathbf{1}_{\sigma}(j))\mathbf{1}_{\sigma}(n)w(j,n)
       - \#_{w(k,\cdot)}(\sigma) - \#_{w(\cdot,k)}(\sigma) - w(k,k)+\sum_{j=0}^{\infty}w(j,k)\Big].
\end{split}
\end{equation*}
It then follows from these two equalities that
\begin{equation*}
\begin{split}
  (1-\mathbf{1}_{\sigma}(k))\vartheta_w(\sigma\cup k)
   &= (1-\mathbf{1}_{\sigma}(k))\Big[\sum_{j=0}^{\infty} \mathbf{1}_{\sigma\cup k}(j)w(j,j)
         + \sum_{j,n=0}^{\infty}(1-\mathbf{1}_{\sigma\cup k}(j))\mathbf{1}_{\sigma\cup k}(n)w(j,n)\Big]\\
   & = (1-\mathbf{1}_{\sigma}(k))\Big[\vartheta_w(\sigma)-\#_{w(k,\cdot)}(\sigma)-\#_{w(\cdot,k)}(\sigma) + \sum_{j=0}^{\infty}w(j,k)\Big]
\end{split}
\end{equation*}
This completes the proof.
\end{proof}

\begin{theorem}\label{thr-2D-commutation-annihilation}
Let $w$ be a 2D-weight and $\mathfrak{K}_w$ the 2D-GWN operator associated with $w$. Then, for all $k\geq 0$, it holds true that
\begin{equation}\label{eq-2D-commutation-annihilation}
  \mathfrak{K}_w\mathfrak{a}_k = \mathfrak{a}_k \mathfrak{K}_w + \mathfrak{a}_k \mathfrak{N}_{w(k,\cdot)}
      + \mathfrak{a}_k\mathfrak{N}_{w(\cdot,k)} -\Big[2w(k,k) + \sum_{j=0}^{\infty}w(j,k)\Big]\mathfrak{a}_k.
\end{equation}
\end{theorem}

\begin{proof}
Let $k\geq 0$. Then, for all $\Phi\in \mathcal{S}^*(M)$ and $\sigma\in \Gamma$,
using properties of the involved operators, we have
\begin{equation*}
  \widehat{\mathfrak{a}_k\mathfrak{K}_w\Phi}(\sigma)
= (1-\mathbf{1}_{\sigma}(k))\widehat{\mathfrak{K}_w\Phi}(\sigma\cup k)
= (1-\mathbf{1}_{\sigma}(k)) \vartheta_w(\sigma\cup k)\widehat{\Phi}(\sigma\cup k),
\end{equation*}
which together with Proposition~\ref{prop-3-12} gives
\begin{equation*}
\begin{split}
  \widehat{\mathfrak{a}_k\mathfrak{K}_w\Phi}(\sigma)
 &= (1-\mathbf{1}_{\sigma}(k))\Big[\vartheta_w(\sigma)-\#_{w(k,\cdot)}(\sigma)-\#_{w(\cdot,k)}(\sigma) + \sum_{j=0}^{\infty}w(j,k)\Big]\widehat{\Phi}(\sigma\cup k)\\
 &= \Big[\vartheta_w(\sigma)-\#_{w(k,\cdot)}(\sigma)-\#_{w(\cdot,k)}(\sigma) + \sum_{j=0}^{\infty}w(j,k)\Big]\widehat{\mathfrak{a}_k\Phi}(\sigma)\\
 &= \widehat{\mathfrak{K}_w\mathfrak{a}_k\Phi}(\sigma)-\widehat{\mathfrak{N}_{w(k,\cdot)}\mathfrak{a}_k\Phi}(\sigma)
    -\widehat{\mathfrak{N}_{w(\cdot,k)}\mathfrak{a}_k\Phi}(\sigma) + \Big[\sum_{j=0}^{\infty}w(j,k)\Big]\widehat{\mathfrak{a}_k\Phi}(\sigma).
\end{split}
\end{equation*}
Thus, by the arbitrariness of $\Phi$ and $\sigma$ in the above equality, we come to the following equality
\begin{equation*}
  \mathfrak{a}_k\mathfrak{K}_w
   = \mathfrak{K}_w\mathfrak{a}_k -\mathfrak{N}_{w(k,\cdot)}\mathfrak{a}_k - \mathfrak{N}_{w(\cdot,k)}\mathfrak{a}_k + \Big[\sum_{j=0}^{\infty}w(j,k)\Big]\mathfrak{a}_k,
\end{equation*}
which, together with $\mathfrak{N}_{w(k,\cdot)}\mathfrak{a}_k = \mathfrak{a}_k\mathfrak{N}_{w(k,\cdot)}- w(k,k)\mathfrak{a}_k$
and $\mathfrak{N}_{w(\cdot,k)}\mathfrak{a}_k = \mathfrak{a}_k\mathfrak{N}_{w(\cdot,k)}-w(k,k)\mathfrak{a}_k$,
(see Theorem~\ref{thr-1D-commutation}),
yields
\begin{equation*}
\mathfrak{a}_k\mathfrak{K}_w
   = \mathfrak{K}_w\mathfrak{a}_k -\mathfrak{a}_k\mathfrak{N}_{w(k,\cdot)}
   - \mathfrak{a}_k\mathfrak{N}_{w(\cdot,k)} + \Big[2w(k,k)+ \sum_{j=0}^{\infty}w(j,k)\Big]\mathfrak{a}_k,
\end{equation*}
which implies (\ref{eq-2D-commutation-annihilation}).
\end{proof}

\begin{proposition}\label{prop-3-14}
Let $w$ be a 2D-weight. Then, for all $k\geq 0$ and $\sigma\in \Gamma$, it holds true that
\begin{equation}\label{eq}
  \mathbf{1}_{\sigma}(k)\vartheta_w(\sigma\setminus k)
   = \mathbf{1}_{\sigma}(k)\Big[\vartheta_w(\sigma)+\#_{w(k,\cdot)}(\sigma)+\#_{w(\cdot,k)}(\sigma) - 2w(k,k)- \sum_{j=0}^{\infty}w(j,k)\Big],
\end{equation}
where $\#_{w(k,\cdot)}(\sigma)= \sum_{n=0}^{\infty}\mathbf{1}_{\sigma}(n)w(k,n)$ and $\#_{w(\cdot,k)}(\sigma)= \sum_{j=0}^{\infty}\mathbf{1}_{\sigma}(j)w(j,k)$.
\end{proposition}

\begin{proof}
The proof is much similar to that of Proposition~\ref{prop-3-12}. Here, we omit it for brevity.
\end{proof}

\begin{theorem}\label{thr-2D-commutation-creation}
Let $w$ be a 2D-weight and $\mathfrak{K}_w$ the 2D-GWN operator associated with $w$. Then, for all $k\geq 0$, it holds true that
\begin{equation}\label{eq-2D-commutation-creation}
  \mathfrak{K}_w\mathfrak{a}_k^{\dag} = \mathfrak{a}_k^{\dag} \mathfrak{K}_w - \mathfrak{a}_k^{\dag} \mathfrak{N}_{w(k,\cdot)}
      - \mathfrak{a}_k^{\dag}\mathfrak{N}_{w(\cdot,k)} + \Big[\sum_{j=0}^{\infty}w(j,k)\Big]\mathfrak{a}_k^{\dag}.
\end{equation}
\end{theorem}

\begin{proof}
Let $k\geq 0$. For all $\Phi\in \mathcal{S}^*(M)$ and $\sigma\in \Gamma$, using properties of the involved operators
as well as Proposition~\ref{prop-3-14}, we get
\begin{equation*}
\begin{split}
  \widehat{\mathfrak{a}_k^{\dag}\mathfrak{K}_w\Phi}(\sigma)
    &= \mathbf{1}_{\sigma}(k)\vartheta_w(\sigma\setminus k)\widehat{\Phi}(\sigma\setminus k)\\
    &= \mathbf{1}_{\sigma}(k)\Big[\vartheta_w(\sigma)+\#_{w(k,\cdot)}(\sigma)+\#_{w(\cdot,k)}(\sigma) - 2w(k,k)- \sum_{j=0}^{\infty}w(j,k)\Big]\widehat{\Phi}(\sigma\setminus k)\\
    &=\Big[\vartheta_w(\sigma)+\#_{w(k,\cdot)}(\sigma)+\#_{w(\cdot,k)}(\sigma) - 2w(k,k)- \sum_{j=0}^{\infty}w(j,k)\Big]\widehat{\mathfrak{a}_k^{\dag}\Phi}(\sigma)\\
    &=\widehat{\mathfrak{K}_w\mathfrak{a}_k^{\dag}\Phi}(\sigma) + \widehat{\mathfrak{N}_{w(k,\cdot)}\mathfrak{a}_k^{\dag}\Phi}(\sigma)
       + \widehat{\mathfrak{N}_{w(\cdot,k)}\mathfrak{a}_k^{\dag}\Phi}(\sigma) - \Big[2w(k,k)+ \sum_{j=0}^{\infty}w(j,k)\Big]\widehat{\mathfrak{a}_k^{\dag}\Phi}(\sigma).
\end{split}
\end{equation*}
Thus, by the arbitrariness of $\Phi$ and $\sigma$ in the above equality, we have
\begin{equation*}
\mathfrak{a}_k^{\dag}\mathfrak{K}_w = \mathfrak{K}_w\mathfrak{a}_k^{\dag} + \mathfrak{N}_{w(k,\cdot)}\mathfrak{a}_k^{\dag} + \mathfrak{N}_{w(\cdot,k)}\mathfrak{a}_k^{\dag}
- \Big[2w(k,k)+ \sum_{j=0}^{\infty}w(j,k)\Big]\mathfrak{a}_k^{\dag},
\end{equation*}
which, together with $\mathfrak{N}_{w(k,\cdot)}\mathfrak{a}_k^{\dag} = \mathfrak{a}_k^{\dag}\mathfrak{N}_{w(k,\cdot)} + w(k,k)\mathfrak{a}_k^{\dag}$
and $\mathfrak{N}_{w(\cdot,k)}\mathfrak{a}_k^{\dag} = \mathfrak{a}_k^{\dag}\mathfrak{N}_{w(\cdot,k)} + w(k,k)\mathfrak{a}_k^{\dag}$,
(see Theorem~\ref{thr-1D-commutation}), gives
\begin{equation*}
  \mathfrak{a}_k^{\dag}\mathfrak{K}_w = \mathfrak{K}_w\mathfrak{a}_k^{\dag} + \mathfrak{a}_k^{\dag} \mathfrak{N}_{w(k,\cdot)}
   + \mathfrak{a}_k^{\dag}\mathfrak{N}_{w(\cdot,k)}- \Big[\sum_{j=0}^{\infty}w(j,k)\Big]\mathfrak{a}_k^{\dag},
\end{equation*}
which is equivalent to (\ref{eq-2D-commutation-creation}).
\end{proof}

\begin{remark}
Again comparing Theorem~\ref{thr-2D-commutation-annihilation}, Theorem~\ref{thr-2D-commutation-creation} and Lemma~\ref{thr-CR-A},
one can find that the commutation relations between $\mathfrak{K}_w$ and $\mathfrak{a}_k$ (respectively, $\mathfrak{a}_k^{\dag}$)
are quite similar to those between $S_w$ and $\partial_k$ (respectively, $\partial_k^*$).
However, the commutation relations between $\mathfrak{K}_w$ and $\mathfrak{a}_k$ (respectively, $\mathfrak{a}_k^{\dag}$) hold true on the whole space $\mathcal{S}^*(M)$,
while the commutation relations between $S_w$ and $\partial_k$ (respectively, $\partial_k^*$) hold true only on $\mathrm{Dom}\,N$,
which is a proper subspace of $\mathcal{L}^2(M)$.
\end{remark}

\begin{theorem}\label{thr-3-17}
Let $w$ be a 2D-weight and $\mathfrak{K}_w$ the 2D-GWN operator associated with $w$. Then, for all $k\geq 0$, it holds true that
\begin{equation}\label{eq-3-30}
  \mathfrak{K}_w\mathfrak{a}_k^{\dag}\mathfrak{a}_k = \mathfrak{a}_k^{\dag}\mathfrak{a}_k \mathfrak{K}_w,
\end{equation}
namely $\mathfrak{K}_w$ commutes with $\mathfrak{a}_k^{\dag}\mathfrak{a}_k$.
\end{theorem}

\begin{proof}
From all $\Phi\in \mathcal{S}^*(M)$ and $\sigma\in \Gamma$, by Proposition~\ref{prop-3-5} we have
$\widehat{\mathfrak{a}_k^{\dag}\mathfrak{a}_k\Phi}(\sigma)= \mathbf{1}_{\sigma}(k)\widehat{\Phi}(\sigma)$,
which together with Theorem~\ref{basic-theorem-1} implies that
\begin{equation*}
  \widehat{\mathfrak{K}_w\mathfrak{a}_k^{\dag}\mathfrak{a}_k\Phi}(\sigma)
 = \vartheta_w(\sigma)\widehat{\mathfrak{a}_k^{\dag}\mathfrak{a}_k\Phi}(\sigma)
= \vartheta_w(\sigma)\mathbf{1}_{\sigma}(k)\widehat{\Phi}(\sigma)
= \mathbf{1}_{\sigma}(k)\widehat{\mathfrak{K}_w\Phi}(\sigma)
= \widehat{\mathfrak{a}_k^{\dag}\mathfrak{a}_k\mathfrak{K}_w\Phi}(\sigma).
\end{equation*}
Thus $\mathfrak{K}_w\mathfrak{a}_k^{\dag}\mathfrak{a}_k = \mathfrak{a}_k^{\dag}\mathfrak{a}_k \mathfrak{K}_w$.
\end{proof}

\section{Appendix}\label{sec-4}

In the appendix, we collect some notions and facts about weighted number (WN) operators in $\mathcal{L}^2(M)$,
which were essentially introduced and proven in \cite {wang-tang} (see Section IV therein).
We continue to use the notation, notions and assumptions made in Section~\ref{sec-2}.

Recall that $\mathcal{L}^2(M)$ is the complex Hilbert space consisting of all square integrable functionals of $M$,
and $\{Z_{\sigma} \mid \sigma\in \Gamma\}$ its canonical orthonormal basis (for details, see (\ref{eq-SquareIntegrable-Space}) and the paragraph standing with it).
For each nonnegative integer $k\geq 0$, there exists a bounded operator $\partial_k$ on $\mathcal{L}^2(M)$ such that
\begin{equation}\label{eq-4-1}
  \partial_kZ_{\sigma} = \mathbf{1}_{\sigma}(k)Z_{\sigma\setminus k},\quad \partial_k^*Z_{\sigma} = (1-\mathbf{1}_{\sigma}(k))Z_{\sigma\cup k},\quad \sigma\in\Gamma,
\end{equation}
where $\sigma\setminus k=\sigma\setminus \{k\}$, $\sigma\cup k=\sigma\cup \{k\}$ and $\partial_k^*$ denotes the adjoint of $\partial_k$.
Usually, $\partial_k$ and $\partial_k^*$ are known as the annihilation and creation operators on $\mathcal{L}^2(M)$ associated with $k$, respectively.

Annihilation and creation operators $\{\partial_k, \partial_k^* \mid k\geq 0\}$ on $\mathcal{L}^2(M)$ satisfy commutation relations of the following form
\begin{equation}\label{eq-4-2}
    \partial_k \partial_l = \partial_l\partial_k,\quad
    \partial_k^{\ast} \partial_l^{\ast} = \partial_l^{\ast}\partial_k^{\ast},\quad
    \partial_k^{\ast} \partial_l = \partial_l\partial_k^{\ast}\quad (k\neq l)
\end{equation}
and
\begin{equation}\label{eq-4-3}
  \partial_k^*\partial_k + \partial_k\partial_k^* =I,
\end{equation}
where $I$ means the identity operator on $\mathcal{L}^2(M)$. In the physical literature, $(\ref{eq-4-3})$ is usually known as the canonical anti-commutation relation
(CAR) in equal time.

A double vector series $\sum_{j,k=0}^{\infty}\xi_{jk}$ in $\mathcal{L}^2(M)$ is said to converge naturally if its sequence of square-array partial sums
\begin{equation*}
   s_n = \sum_{j,k=0}^n\xi_{jk}= \sum_{j=0}^n\sum_{k=0}^n\xi_{jk},\quad n\geq 0
\end{equation*}
converges in $\mathcal{L}^2(M)$ as $n\rightarrow \infty$. In that case, one writes $\sum_{j,k=0}^{\infty}\xi_{jk}= \lim_{n\to \infty}s_n$.
Let $w$ be a 2D-weight (see Definition~\ref{def-3-1} for its definition).
Then, for each $\sigma\in \Gamma$, the following double vector series converges naturally
\begin{equation}\label{eq-4-4}
   \sum_{j,k=0}^{\infty} w(j,k)\partial_k^*\partial_j\partial_j^*\partial_kZ_{\sigma}.
\end{equation}

Given a 2D-weight $w$, the 2D-weighted number operator (for short, the 2D-WN operator or the WN operator) $S_w$ associated with $w$
is the one in $\mathcal{L}^2(M)$ defined by
\begin{equation}\label{eq-4-5}
  S_w\xi=\sum_{j,k=0}^{\infty}w(j,k)\partial_k^*\partial_j\partial_j^*\partial_k\xi,\quad \xi \in \mathrm{Dom}\, S_w
\end{equation}
with
\begin{equation}\label{eq-4-6}
  \mathrm{Dom}\, S_w =\Big\{\, \xi \in \mathcal{L}^2(M) \Bigm| \sum_{j,k=0}^{\infty}w(j,k)\partial_k^*\partial_j\partial_j^*\partial_k\xi\ \ \mbox{converges naturally}               \,\Big\}.
\end{equation}
It follows immediately from (\ref{eq-4-5}) that $\{Z_{\sigma} \mid \sigma\in \Gamma\}\subset \mathrm{Dom}\, S_w$,
which implies that $S_w$ is densely-defined in $\mathcal{L}^2(M)$.

\begin{lemma} (Theorem~4.2 and Theorem~4.3 of \cite{wang-tang})\label{lem-4-1}\ \
Let $w$ be a 2D-weight. Then, a vector $\xi \in \mathcal{L}^2(M)$ falls into $\mathrm{Dom}\, S_w$ if and only if it satisfies that
\begin{equation}\label{eq-4-7}
  \sum_{\sigma \in \Gamma}\big(\vartheta_w(\sigma)\big)^2|\langle Z_{\sigma},\xi\rangle|^2<\infty,
\end{equation}
where $\vartheta_w$ is the spectral function associated with $w$ (see Definition~\ref{def-3-1}).
In that case, one has
\begin{equation}\label{eq-4-8}
  S_w\xi = \sum_{\sigma \in \Gamma}\vartheta_w(\sigma)\langle Z_{\sigma},\xi\rangle Z_{\sigma}.
\end{equation}
Additionally, $S_w$ is bounded if and only if $\sup_{\sigma \in \Gamma}\vartheta_w(\sigma)<\infty$.
\end{lemma}

\begin{remark}\label{rem-4-1}
Due to the above lemma, the 2D-WN operator $S_w$ associated with a 2D-weight $w$ can be equivalently redefined as
\begin{equation}\label{eq-weighted-operator-in-L2}
  S_w\xi = \sum_{\sigma\in \Gamma} \vartheta_w(\sigma)\langle Z_{\sigma},\xi\rangle Z_{\sigma},\quad \xi \in \mathrm{Dom}\,S_w
\end{equation}
with
\begin{equation}\label{eq-4-10}
  \mathrm{Dom}\,S_w=\Big\{\xi \in \mathcal{L}^2(M) \Bigm| \sum_{\sigma\in \Gamma} (\vartheta_w(\sigma))^2|\langle Z_{\sigma},\xi\rangle|^2<\infty \Big\},
\end{equation}
which implies that $S_w$ is positive and self-adjoint.
\end{remark}

For a bounded nonnegative function $u$ on $\mathbb{N}$, the 1D-weighted number operator (for short, the 1D-WN operator)
$N_u$ associated with $u$ is the one in $\mathcal{L}^2(M)$ defined by
\begin{equation}\label{eq-4-11}
  N_u\xi = \sum_{\sigma \in \Gamma}\#_u(\sigma)\langle Z_{\sigma},\xi\rangle Z_{\sigma},\quad \mathrm{Dom}\,N_u
\end{equation}
with
\begin{equation}\label{eq-4-12}
 \mathrm{Dom}\,N_u
 = \Big\{\xi \in  \mathcal{L}^2(M) \Bigm| \sum_{\sigma \in \Gamma}\big(\#_u(\sigma)\big)^2|\langle Z_{\sigma},\xi\rangle|^2<\infty\Big\},
\end{equation}
where $\#_u(\sigma)=\sum_{k=0}^{\infty}\mathbf{1}_{\sigma}(k)u(k)$.

Let $u$ be a bounded nonnegative function on $\mathbb{N}$. Then, there exists a 2D-weight $w^{(u)}$ on $\mathbb{N}\times \mathbb{N}$ such that
\begin{equation}\label{eq-4-13}
  w^{(u)}(j,k)=
  \left\{
    \begin{array}{ll}
      u(j), & \hbox{$j=k$, $(j,k)\in \mathbb{N}\times \mathbb{N}$;} \\
      0, & \hbox{$j\neq k$, $(j,k)\in \mathbb{N}\times \mathbb{N}$}
    \end{array}
  \right.
\end{equation}
and $N_u= S_{w^{(u)}}$, which together with (\ref{eq-4-4}) gives a representation of $N_u$ of the following form
\begin{equation}\label{eq-4-14}
 N_u\xi=\sum_{k=0}^{\infty}u(k)\partial_k^*\partial_k\xi,\quad \xi \in \mathrm{Dom}\, N_u.
\end{equation}
Thus, 1D-WN operators are actually a special class of 2D-WN operators.

The 1D-WN operator $N_u$ associated with the constant function $u(k)\equiv 1$ is usually known as the number operator
and is denoted by $N$. Clearly, the number operator $N$ can also be directly defined as
\begin{equation}\label{eq-4-15}
  N\xi =  \sum_{\sigma \in \Gamma}\#(\sigma)\langle Z_{\sigma},\xi\rangle Z_{\sigma},\quad \mathrm{Dom}\,N,
\end{equation}
where the domain $\mathrm{Dom}\,N$ is given by
\begin{equation}\label{eq-4-16}
   \mathrm{Dom}\,N
 = \Big\{\xi \in  \mathcal{L}^2(M) \Bigm| \sum_{\sigma \in \Gamma}\big(\#(\sigma)\big)^2|\langle Z_{\sigma},\xi\rangle|^2<\infty\Big\}.
\end{equation}
The next result shows that the domain of the number operator $N$ can play an important role
in comparing two WN operators.

\begin{lemma}(Theorem~4.6 of \cite{wang-tang})\label{thr-4-6}\ \
Let $w$ be a 2D-weight. Then $\mathrm{Dom}\, N$ is a core for $S_w$.
In particular, for all bounded nonnegative function $u$ on $\mathbb{N}$, $\mathrm{Dom}\, N$ is a core for $N_u$.
\end{lemma}

From (\ref{eq-4-5}), one can see that 2D-WN operators are essentially generated by annihilation and creation operators
$\{\partial_k, \partial_k^* \mid k\geq 0\}$ on $\mathcal{L}^2(M)$. This, together with the commutation relations described by (\ref{eq-4-2}) and (\ref{eq-4-3}),
suggests that a 2D-WN operator $S_w$ and an annihilation operator $\partial_k$ (or a creation operator $\partial_k^*$) must satisfy
some kind of meaningful commutation relations. This is indeed the case.

\begin{lemma}(Theorem 4.7 and Theorem 4.8 of \cite{wang-tang})\label{lem-4-3}\ \
Let $u$ be a bounded nonnegative function on $\mathbb{N}$. Then,
for all $n\geq 0$, $\mathrm{Dom}\,N_u$ is an invariant subspace of both $\partial_n$ and $\partial_n^*$, and moreover
it holds true on $\mathrm{Dom}\,N_u$ that
\begin{equation}\label{eq-4-17}
  N_u\partial_n = \partial_nN_u-u(n)\partial_n, \quad N_u\partial_n^*= \partial_n^*N_u+ u(n)\partial_n^*.
\end{equation}
\end{lemma}

Let $w$ be a 2D-weight. Then, for each $n\geq 0$, both $w(n,\cdot)$ and $w(\cdot,n)$ are bounded nonnegative functions on $\mathbb{N}$,
hence $N_{w(n,\cdot)}$ and $N_{w(\cdot,n)}$ make sense as the 1D-WN operators associated with $w(n,\cdot)$ and $w(\cdot,n)$ respectively.

\begin{lemma}(Theorem 4.10 and Theorem 4.11 of \cite{wang-tang})\label{thr-CR-A}\ \
Let $w$ be a 2D-weight. Then, for all $n\geq 0$, it holds true on $\mathrm{Dom}\, N$ that
\begin{equation}\label{eq-CR-A}
  S_w\partial_n
      = \partial_nS_w
        + \partial_nN_{w(\cdot,n)}
    +\partial_nN_{w(n,\cdot)} -\Big[ 2w(n,n) +\sum_{j=0}^{\infty}w(j,n)\Big]\partial_n
\end{equation}
and
\begin{equation}\label{eq-CR-C}
  S_w\partial_n^* = \partial_n^*S_w  - \partial_n^* N_{w(\cdot,n)}- \partial_n^*N_{w(n,\cdot)}
   + \Big(\sum_{j=0}^{\infty}w(j,n)\Big)\partial_n^*.
\end{equation}
\end{lemma}

It should be emphasized that, in general, the commutation relations described by (\ref{eq-4-17}), (\ref{eq-CR-A}) and (\ref{eq-CR-C}) make no sense on the whole space $\mathcal{L}^2(M)$.

\section*{Acknowledgement}

The authors are extremely grateful to the referees for their valuable comments and suggestions on improvement of the first version of the present paper.
This work is supported by National Natural Science Foundation of China (Grant No. 11861057, 12261080).

\end{document}